\documentclass[11pt]{amsart}
\usepackage{amssymb, amsthm, amsmath, amsfonts}
\usepackage{graphics}
\usepackage{hyperref}
\usepackage[all]{xy}
\usepackage{enumerate}
\usepackage[mathscr]{eucal}
\usepackage{calrsfs, bm}
\usepackage{thmtools}
\usepackage{multirow}

\declaretheorem[numberlike=equation]{theorem}
\declaretheorem[name=Lemma, sibling=theorem]{lemma}
\declaretheorem[name=Proposition, sibling=theorem]{proposition}

\declaretheorem[name=Fact, sibling=theorem]{fact}
\declaretheorem[name=Definition, style=definition, sibling=theorem]{definition}
\declaretheorem[name=Notation, style=definition, sibling=theorem]{notation}
\declaretheorem[name=Remark, style=definition, sibling=theorem]{remark}

\DeclareMathOperator{\Hom}{Hom}

\DeclareMathOperator{\Res}{Res}
\DeclareMathOperator{\Ind}{Ind}

\newcommand{\alp}{{\pmb{\bm\alpha}}}
\newcommand{\be}{{\pmb{\bm\beta}}}
\newcommand{\CC}{{\mathcal{C}}}
\newcommand{\F}{{\mathbb{F}}}
\newcommand{\FI}{{\mathrm{FI}}}
\newcommand{\GL}{{\mathrm{GL}}}
\newcommand{\kk}{{\Bbbk}}
\newcommand{\la}{{\pmb{\bm\lambda}}}
\newcommand{\muu}{{\pmb{\bm\mu}}}
\newcommand{\nuu}{{\pmb{\bm\nu}}}
\newcommand{\PP}{{\mathcal{P}}}
\newcommand{\VI}{{\mathrm{VI}}}
\newcommand{\VIC}{{\mathrm{VIC}}}
\newcommand{\Z}{{\mathbb{Z}}}
\newcommand{\ZZ}{{\mathrm{ZZ}}}

\newcommand{\vdotrows}[2][-15pt]{%
  \multirow{#2}{*}{%
    \vbox {%
      \baselineskip = \dimexpr 2pt\relax 
      \multiply\baselineskip by #2\relax
      \advance\baselineskip by 2pt\relax
      \lineskiplimit = 0pt\relax
      \kern 6pt\relax
      \vskip #1\relax%
      \hbox {.}\hbox {.}\hbox {.}}%
  }%
}

\title[Stable decompositions of certain representations]{Stable decompositions of certain representations of the finite general linear groups}

\author{Wee Liang Gan}
\address{Department of Mathematics, University of California, Riverside, CA 92521, USA}
\email{wlgan@math.ucr.edu}

\author{John Watterlond}
\address{Department of Mathematics, University of California, Riverside, CA 92521, USA}
\email{john.watterlond@gmail.com}

\begin{document}

\begin{abstract}
We prove that the irreducible decomposition of the permutation representation of $\GL_n(\F_q)$ on $\GL_n(\F_q)/\GL_{n-m}(\F_q)$ stabilizes for large $n$. We deduce, as a consequence, a representation stability theorem for finitely generated VIC-modules.
\end{abstract}

\maketitle

\section{Introduction} \label{introduction}

Let $\kk$ be an algebraically closed field of characteristic zero and $\F_q$ a finite field of order $q$. Denote by $G_n$ the finite general linear group $\GL_n(\F_q)$. For $r<n$, we consider $G_r$ as a subgroup of $G_n$ in the standard way, and write $\kk[G_n/G_r]$ for the permutation representation of $G_n$ on the set $G_n/G_r$. In this paper, we show that the collection of multiplicities which occur in the irreducible decomposition of $\kk[G_n/G_{n-m}]$ as a representation of $G_n$ is independent of $n$ for $n\geqslant 3m$. We prove, in fact, a more precise theorem that describes how the irreducible decomposition is independent of $n$ for $n\geqslant 3m$. To state the theorem, we need to recall the parametrization of the irreducible representations of $G_n$ over $\kk$.

The irreducible representations of $G_n$ were found by Green \cite{Green}. Let $\CC_n$ be the set of cuspidal irreducible representations of $G_n$ (up to isomorphism), and let $\CC = \sqcup_{n\geqslant 1} \CC_n$. We set $\mathrm{d}(\rho)=n$ if $\rho\in \CC_n$. By a partition, we mean a non-increasing sequence of non-negative integers $(\lambda_1, \lambda_2, \ldots)$ where only finitely many terms are non-zero. If $\lambda=(\lambda_1, \lambda_2, \ldots)$ is a partition, we set $|\lambda|=\lambda_1 + \lambda_2 + \cdots$. Let $\PP$ be the set of partitions. For any function $\muu: \CC \to \PP$, let
\begin{equation*}
\| \muu \| = \sum_{\rho\in\CC} \mathrm{d}(\rho) |\muu(\rho)| . 
\end{equation*}
It is well-known (see, for example, \cite{SZ} or \cite{Zelevinsky}) that there is a natural parametrization of the isomorphism classes of irreducible representations of $G_n$ by functions $\muu: \CC \to \PP$ such that $\| \muu \| = n$. We write $\varphi(\muu)$ for the irreducible representation of $G_n$ parametrized by $\muu$.

Let $\iota\in\CC_1$ be the trivial representation of $G_1$. Suppose $\la: \CC \to \PP$ is a function and $\la(\iota) = (\lambda_1,\lambda_2, \ldots)$. If $n$ is an integer $\geqslant \|\la\| + \lambda_1$, we define the function $\la[n] : \CC \to \PP$ with $\|\la[n]\|=n$ by 
\begin{equation*}
\la[n](\rho) = \left\{ \begin{array}{ll}
(n-\|\la\|, \lambda_1, \lambda_2, \ldots) & \mbox{ if } \rho=\iota, \\
\la(\rho) & \mbox{ if } \rho\neq\iota.
\end{array} \right.
\end{equation*}
It is plain that for each function $\muu: \CC\to\PP$ with $\|\muu\| < \infty$, there exists a unique function $\la: \CC\to\PP$ such that $\muu=\la[n]$ where $n=\|\muu\|$.

\begin{definition}
Suppose $V_m, V_{m+1}, \ldots$ is a sequence where each $V_n$ is a representation of $G_n$ over $\kk$. We say that the sequence $\{V_n\}$ is \emph{multiplicity stable} if there exists an integer $N\geqslant m$ such that  in the irreducible decomposition
\begin{equation*}
V_n  = \bigoplus_{\la} \varphi(\la[n])^{\oplus c_n(\la)} \quad\quad \mbox{ (where $0\leqslant c_n(\la) \leqslant \infty$)},
\end{equation*}
the multiplicities $c_n(\la)$ do not depend on $n$ for $n\geqslant N$; in particular, for any $\la$ such that $\la[N]$ is not defined, one has $c_n(\la)=0$ for every $n\geqslant N$. We call the smallest such $N$ the \emph{multiplicity stability degree} of $\{V_n\}$.
\end{definition}

\begin{theorem} \label{main theorem}
Fix a non-negative integer $m$, and let $V_n = \kk[G_n/G_{n-m}]$ for $n\geqslant m$. The sequence $\{V_n\}$ is multiplicity stable with multiplicity stability degree $\leqslant 3m$.
\end{theorem}

Our motivation to consider the question of multiplicity stability for the sequence in the above theorem comes from representation stability theory, in which a similar phenomenon for representations of symmetric groups is studied in \cite{CEF} and \cite{CF}; see \cite{Farb} for a survey. In particular, an analog of Theorem \ref{main theorem} for symmetric groups was proved by Hemmer \cite[Theorem 2.4]{Hemmer} using Pieri's rule. Our proof of Theorem \ref{main theorem} uses a  branching rule for the finite general linear groups due to Thoma \cite[Satz 2]{Thoma} and Zelevinsky \cite[Corollary 13.8]{Zelevinsky}.

In more detail, suppose $V_0 \stackrel{\phi_0}{\longrightarrow} V_1  \stackrel{\phi_1}{\longrightarrow} \cdots$ is a sequence where each $V_n$ is a representation of $G_n$ over $\kk$ and each $\phi_n$ is a $\kk$-linear map. We say that $\{V_n, \phi_n\}$ is a \emph{consistent sequence} if, for every non-negative integer $n$ and for every $g\in G_n$, the following diagram commutes:
\begin{equation*}
\xymatrix{ V_n \ar[r]^{\phi_n} \ar[d]_{g} & V_{n+1} \ar[d]^{g} \\ V_n \ar[r]_{\phi_n} & V_{n+1} }
\end{equation*}
(where $g$ acts on $V_{n+1}$ by considering it as an element of $G_{n+1}$). In analogy with \cite[Definition 2.6]{CF}, we made the following definition in \cite[Definition 1.5]{GW}.

\begin{definition}
A consistent sequence $\{V_n, \phi_n\}$ is \emph{representation stable} if the following conditions are satisfied:
\begin{itemize}
\item[(RS1)]
For all $n$ sufficiently large, the map $\phi_n : V_n \to V_{n+1}$ is injective.

\item[(RS2)]
For all $n$ sufficiently large, the span of the $G_{n+1}$-orbit of $\phi_n(V_n)$ is all of $V_{n+1}$.

\item[(RS3)]
The sequence $\{V_n\}$ is multiplicity stable.
\end{itemize}
\end{definition}

Let $\VIC$ be the category whose objects are the finite dimensional vector spaces over $\F_q$ and whose morphisms are the pairs $(f,K)$ where $f$ is an injective linear map and $K$ is a complementary subspace to the image of $f$. More precisely, a morphism $X\to Y$ in the category $\VIC$ is the data of an injective linear map $f:X\to Y$ and a subspace $K$ of $Y$ such that $Y$ is the internal direct sum of $f(X)$ and $K$.
The composition of morphisms is defined in the natural way: $(g, L)\circ (f,K) = (g\circ f,\,g(K)\oplus L)$. Observe that every endomorphism in the category $\VIC$ is an automorphism, and the group of automorphisms of $\F_q^{\,n}$ as an object of $\VIC$ is the group $G_n$. 

A \emph{$\VIC$-module} over a commutative ring $R$ is, by definition, a functor from $\VIC$ to the category of $R$-modules. We shall recall the notion of finite generation for $\VIC$-modules in Section \ref{section VIC}.
It should be noted that $\VIC$-modules were first introduced and studied by Putman and Sam in \cite{PS}; they proved, in particular, that finitely generated $\VIC$-modules over a noetherian ring has an inductive description named central stability (see \cite[Theorem E]{PS}). For a finitely generated $\VIC$-module $V$ over $\kk$, our next result gives a description of the asymptotic behaviour of the sequence $\{V(\F_q^{\,n})\}$ as representations of the groups $G_n$. A $\VIC$-module $V$ over $\kk$ defines a consistent sequence $\{V_n, \phi_n\}$ where $V_n=V(\F_q^{\,n})$ and $\phi_n$ is induced by the morphism $(f_n, K_n)$ where $f_n$ is the standard inclusion $\F_q^{\,n} \hookrightarrow \F_q^{\,n+1}$ and $K_n$ is the subspace of vectors in $\F_q^{\,n+1}$ whose first $n$ coordinates are 0.

\begin{theorem} \label{stability for VIC modules}
Let $V$ be a $\VIC$-module over $\kk$ and $\{V_n, \phi_n\}$ the consistent sequence obtained from $V$. Then $V$ is finitely generated if and only if $\{V_n, \phi_n\}$ is representation stable and $\dim(V_n)<\infty$ for every $n$. Moreover, if $V$ is finitely generated, then there exists a polynomial $P\in \mathbb{Q}[T]$ and an integer $N$ such that $\dim(V_n) = P(q^n)$ for all $n\geqslant N$.
\end{theorem}

As we shall explain in Section \ref{section VIC}, Theorem \ref{stability for VIC modules} generalizes the main results of \cite{GW}. However, the proof of Theorem \ref{stability for VIC modules} depends crucially on the key propositions in \cite{GW} and also uses Theorem \ref{main theorem}.  For each $m\geqslant 0$, the sequence $\kk[G_n/G_{n-m}]$ where $n\geqslant m$ can be assembled into a free $\VIC$-module $P(m)$, and any finitely generated $\VIC$-module over $\kk$ is a quotient of a finite direct sum of $\VIC$-modules of the form $P(m)$.

\begin{remark}
The definition of multiplicity stable (and hence the definition of representation stable) made use of the representation theory of the finite general linear groups over an algebraically closed field of characteristic zero. If $V$ is a finitely generated $\VIC$-module over a field of characteristic zero which is not necessarily algebraically closed, it follows from Theorem \ref{stability for VIC modules} (by extension of scalars) that the dimension of $V_n$ is eventually a polynomial in $q^n$ with rational coefficients. On the other hand, if $d$ is a positive integer, the $\VIC$-module $V$ over $\F_q$ defined by $V(X) = X^{\otimes d}$ is finitely generated but has $\dim(V_n) = n^d$.
\end{remark}

\subsection*{Acknowledgments}
The first author thanks Wee Teck Gan for useful discussions. We are also grateful to the referees for their suggestions to improve the paper.

\section{Branching rule for finite general linear groups} \label{section branching rule}

\subsection{Branching rule} 
The branching rule of Thoma \cite{Thoma} for the groups $G_n$ gives a formula for the multiplicity of an irreducible representation of $G_{n-1}$ in the restriction of an irreducible representation of $G_n$ to $G_{n-1}$. We shall use an equivalent combinatorial description of the multiplicity due to Zelevinsky \cite{Zelevinsky}. 

First, we introduce some notations.

\begin{notation}
Suppose that $\lambda = (\lambda_1, \lambda_2, \ldots)$ and $\mu = (\mu_1, \mu_2, \ldots)$ are partitions. 

We write $\lambda \stackrel{+}{\longrightarrow} \mu$ if 
\begin{equation*}
\lambda_i \leqslant \mu_i \leqslant \lambda_i+1 \quad \mbox{ for every }i, 
\end{equation*}
or equivalently, if the Young diagram of $\mu$ can be obtained by adding at most one box to each row in the Young diagram of $\lambda$. 

We write $\mu \stackrel{-}{\longrightarrow} \lambda$ if 
\begin{equation*}
\mu_i - 1 \leqslant \lambda_i \leqslant \mu_i \quad \mbox{ for every }i,
\end{equation*}
or equivalently, if the Young diagram of $\lambda$ can be obtained by removing at most one box from each row in the Young diagram of $\mu$.
\end{notation}

Obviously, one has $\lambda \stackrel{+}{\longrightarrow} \mu$ if and only if $\mu \stackrel{-}{\longrightarrow} \lambda$. 

\begin{notation}
Suppose we have functions $\la : \CC \to \PP$ and $\muu : \CC \to \PP$. We write $\la\stackrel{+}{\longrightarrow} \muu$ if  $\la(\rho) \stackrel{+}{\longrightarrow} \muu(\rho)$ for each $\rho \in \CC$. Similarly, we write $\muu \stackrel{-}{\longrightarrow} \la$ if $\muu(\rho) \stackrel{-}{\longrightarrow} \la(\rho)$ for each $\rho \in \CC$.
\end{notation}

Recall that we write $\varphi(\muu)$ for the irreducible representation of $G_n$ parametrized by a function $\muu:\CC \to \PP$ such that $\|\muu\|=n$. The following branching rule was proved by Zelevinsky \cite[Corollary 13.8]{Zelevinsky}; it was also proved by Thoma \cite[Satz 2]{Thoma} in another form.

\begin{fact}[Branching rule] \label{branching rule for 1 step}
If $\muu: \CC\to\PP$ is a function such that $\|\muu\|=n$, and $\nuu: \CC\to\PP$ is a function such that $\|\nuu\|=n-1$, then the multiplicity of $\varphi(\nuu)$ in the restriction of $\varphi(\muu)$ to $G_{n-1}$ is equal to the number of functions $\la: \CC\to\PP$ such that $\nuu \stackrel{-}{\longrightarrow} \la$ and $\la \stackrel{+}{\longrightarrow} \muu$.
\end{fact}

We shall need the branching rule for restriction of an irreducible representation of $G_n$ to $G_{n-m}$ where $n\geqslant m\geqslant 1$.

\begin{definition}
Suppose $\muu: \CC\to\PP$ is a function such that $\|\muu\|=n$, and $\nuu: \CC\to\PP$ is a function such that $\|\nuu\|=n-m$, where $n\geqslant m\geqslant 1$. By a \emph{zigzag path} from $\nuu$ to $\muu$, we mean two sequences of functions 
\begin{gather*}
\la^{(r)} : \CC\to\PP \quad \mbox{ where } r=1, \ldots, m,\\
\muu^{(s)} : \CC\to\PP \quad \mbox{ where } s=1, \ldots, m,
\end{gather*}
such that $\|\muu^{(s)}\| = n-m+s$ for every $s$, and one has:
\begin{equation} \label{zigzag path}
\xymatrix{
 && \muu^{(m)}=\muu \\
 \la^{(m)} \ar[rru]^-{+} && \\
 && \muu^{(m-1)} \ar[llu]_-{-} \\
  \ar[rru]^-{+} && \\
 && \\
 &&  \ar@{}[lluu]_{\vdotrows{3}} \\
 \la^{(2)} \ar[rru]^-{+} && \\
 && \muu^{(1)} \ar[llu]_-{-} \\
 \la^{(1)} \ar[rru]^-{+} && \\
 && \nuu \ar[llu]_-{-} 
}.
\end{equation}
We define $\ZZ(\nuu, \muu)$ to be the set of all zigzag paths from $\nuu$ to $\muu$.
\end{definition}

The following theorem follows easily from Fact \ref{branching rule for 1 step}.

\begin{theorem} \label{branching rule for m steps}
Suppose $\muu: \CC\to\PP$ is a function such that $\|\muu\|=n$, and $\nuu: \CC\to\PP$ is a function such that $\|\nuu\|=n-m$, where $n\geqslant m\geqslant 1$. Then the multiplicity of $\varphi(\nuu)$ in the restriction of $\varphi(\muu)$ to $G_{n-m}$ is equal to $|\ZZ(\nuu, \muu)|$.
\end{theorem}
\begin{proof}
We use induction on $m$. The base case $m=1$ is precisely Fact \ref{branching rule for 1 step}. 

Suppose $m>1$. Let $S$ be the set of all functions $\muu': \CC\to\PP$ such that $\|\muu'\|=n-1$. By Fact \ref{branching rule for 1 step}, we have:
\begin{equation*}
\Res^{G_n}_{G_{n-1}} (\varphi(\muu)) = \bigoplus_{\muu'\in S} \left(\bigoplus_{Z\in \ZZ(\muu',\muu)} \varphi(\muu')\right).
\end{equation*}
Hence, by the induction hypothesis, the multiplicity of $\varphi(\nuu)$ in $\Res_{G_{n-m}}$ is equal to:
\begin{equation*}
\sum_{\muu'\in S} \left( \sum_{Z\in \ZZ(\muu',\muu)} |\ZZ(\nuu,\muu')| \right) = |\ZZ(\nuu, \muu)|.
\end{equation*}
\end{proof}

\subsection{Proof of Theorem \ref{main theorem}}
We now prove Theorem \ref{main theorem}.

Fix a non-negative integer $m$ and set $V_n = \kk[G_n/G_{n-m}]$ for each $n\geqslant m$. Observe that:
\begin{equation*}
V_n \cong \Ind^{G_n}_{G_{n-m}} (\iota_{n-m}),
\end{equation*}
where $\iota_{n-m}$ denotes the trivial representation of $G_{n-m}$. Let $\nuu_{n-m}: \CC\to\PP$ be the function defined by 
\begin{equation*}
\nuu_{n-m} = \left\{ \begin{array}{ll}
(n-m, 0, 0, \ldots) & \mbox{ if } \rho=\iota, \\
(0, 0, \ldots) & \mbox{ if } \rho\neq\iota.
\end{array} \right.
\end{equation*}
Then by \cite[Proposition 9.6]{Zelevinsky} one has $\iota_{n-m} = \varphi(\nuu_{n-m})$. 

Suppose $\muu: \CC\to\PP$ is a function such that $\|\muu\|=n$. Then one has:
\begin{align*}
&\mbox{multiplicity of $\varphi(\muu)$ in $V_n$} & \\
=& \dim \Hom_{G_n} (\Ind^{G_n}_{G_{n-m}} (\iota_{n-m}), \varphi(\muu)) & \\
=& \dim \Hom_{G_{n-m}}(\iota_{n-m}, \Res^{G_n}_{G_{n-m}}(\varphi(\muu)) ) & \mbox{(by Frobenius reciprocity)} \\
=& \mbox{multiplicity of $\varphi(\nuu_{n-m})$ in $\Res^{G_n}_{G_{n-m}}(\varphi(\muu))$} & \\
=& |\ZZ(\nuu_{n-m}, \muu)| & \mbox{(by Theorem \ref{branching rule for m steps})}
\end{align*}

Let $\la: \CC\to\PP$ be the function such that $\muu=\la[n]$, and write $\la[n](\iota)$ as $(n-\|\la\|, \lambda_1, \lambda_2, \ldots)$. Suppose that $\varphi(\la[n])$ is an irreducible component of $V_n$. Then there exists a zigzag path from $\nuu_{n-m}$ to $\la[n]$. Observe that when we move along an arrow of type $\stackrel{+}{\longrightarrow}$ in a zigzag path, the number of boxes in a row of the Young diagram  will not decrease and will increase by at most 1; when we move along an arrow of type $\stackrel{-}{\longrightarrow}$, the number of boxes in a row of the Young diagram will not increase and will decrease by at most 1. Hence, one has $n-\|\la\| \geqslant (n-m)-m$ and $\lambda_1\leqslant m$, so $3m - \|\la\| \geqslant \lambda_1$. Therefore, the function $\la[3m] : \CC\to\PP$ is well-defined. It suffices to prove that: 
\begin{equation} \label{equation stability of multiplicities}
|\ZZ(\nuu_{\ell-m}, \la[\ell])| = |\ZZ(\nuu_{\ell+1-m}, \la[\ell+1])| \quad \mbox{ for every } \ell \geqslant 3m.
\end{equation}{

To this end, define a map 
\begin{equation*}
\hbar : \PP \longrightarrow \PP, \quad (\alpha_1, \alpha_2, \alpha_3, \ldots) \mapsto (\alpha_1+1, \alpha_2, \alpha_3, \ldots).
\end{equation*}
For any function $\alp: \CC\to\PP$, define the function $\widetilde\alp: \CC\to\PP$ by $\widetilde\alp(\iota)=\hbar(\alp(\iota))$ and $\widetilde\alp(\rho) = \alp(\rho)$ if $\rho\neq\iota$. Note that $\widetilde{\nuu_{\ell-m}}=\nuu_{\ell+1-m}$ and $\widetilde{\la[\ell]} = \la[\ell+1]$. Moreover, for any functions $\alp:\CC\to\PP$ and $\be:\CC\to\PP$, if $\alp\stackrel{+}{\longrightarrow}\be$ or  $\alp\stackrel{-}{\longrightarrow}\be$, then $\widetilde\alp\stackrel{+}{\longrightarrow}\widetilde\be$ or  $\widetilde\alp\stackrel{-}{\longrightarrow}\widetilde\be$, respectively. Fix $\ell\geqslant 3m$ and define the map
\begin{align*}
h: \ZZ(\nuu_{\ell-m}, \la[\ell]) &\longrightarrow \ZZ(\nuu_{\ell+1-m}, \la[\ell+1]),\\
\left(\{\la^{(r)}\}, \{\muu^{(s)}\}\right) &\longmapsto \left(\{\widetilde{\la^{(r)}}\}, \{\widetilde{\muu^{(s)}}\}\right).
\end{align*}

We claim that the map $h$ is bijective. The injectivity of $h$ is clear from the injectivity of $\hbar$. To see the surjectivity of $h$, note that the image of $\hbar$ is the set of all partitions $(\alpha_1, \alpha_2, \ldots)$ such that $\alpha_1-1\geqslant \alpha_2$. Suppose we have a zigzag path 
\begin{equation*}
\left(\{ \alp^{(r)} \}, \{ \be^{(s)} \}\right) \in \ZZ(\nuu_{\ell+1-m}, \la[\ell+1]).
\end{equation*}
Let us write $\alp^{(r)}(\iota) = (\alpha^{(r)}_1, \alpha^{(r)}_2, \ldots)$ and $\be^{(s)}(\iota) = (\beta^{(s)}_1, \beta^{(s)}_2, \ldots)$. We have:
\begin{gather*}
\alpha^{(r)}_1 \geqslant (\ell+1-m) - r \quad\mbox{ and }\quad \alpha^{(r)}_2 \leqslant r-1;\\
\beta^{(s)}_1 \geqslant (\ell+1-m) - s  \quad\mbox{ and }\quad  \beta^{(s)}_2  \leqslant s.
\end{gather*}
Since $\ell\geqslant 3m$ and $r, s\leqslant m$, it follows that:
\begin{gather*}
\alpha^{(r)}_1 - 1 \geqslant m > \alpha^{(r)}_2;\\
\beta^{(s)}_1 - 1 \geqslant m \geqslant \beta^{(s)}_2.
\end{gather*}
Hence, the zigzag path $\left(\{ \alp^{(r)} \}, \{ \be^{(s)} \}\right)$ is in the image of $h$. The map $h$ is thus bijective and we have proven \eqref{equation stability of multiplicities}. This completes the proof of Theorem \ref{main theorem}.

\section{Representation stability for VIC-modules} \label{section VIC}

\subsection{Notations}
For any finite group $G$ and representation $\pi$ of $G$, we write $\pi_G$ for the quotient space of $G$-coinvariants of $\pi$.

Let $n$ be any non-negative integer. For any non-negative integers $m$ and $r$ such that $n=m+r$, let $H_{m,r}$ be the subgroup of $G_n$ consisting of all matrices of the form:
\begin{equation} \label{matrix g}
g = \left( \begin{array}{cc} 1_m & g_{12} \\ 0 & g_{22} \end{array} \right), 
\end{equation}
and $1_m$ is the identity element of $G_m$ and $g_{22}\in G_r$. Let $L_{m,r}$ be the subgroup of $H_{m,r}$ consisting of all matrices $g$ of the form \eqref{matrix g} such that $g_{12}=0$. Clearly, $L_{m,r}$ is conjugate to $G_r$ in $G_n$.

For any non-negative integer $m$, we define a $\VIC$-module $P(m)$ by
\begin{equation*}
P(m)(-) = \kk \Hom_{\VIC} (\F_q^{\,m}, -),
\end{equation*}
that is, $P(m)$ is the composition of the functor $\Hom_{\VIC} (\F_q^{\,m}, -)$ followed by the free $\kk$-module functor.

\subsection{Generalities}
We recall some basic definitions on $\VIC$-modules. 

A \emph{homomorphism} $F:U\to V$ of $\VIC$-modules is a natural transformation from the functor $U$ to the functor $V$. If $U$ and $V$ are $\VIC$-modules such that $U(X)$ is a $\kk$-submodule of $V(X)$ for every object $X$ of $\VIC$, and the collection of inclusion maps $U(X) \hookrightarrow V(X)$ defines a homomorphism $U\to V$ of $\VIC$-modules, then we say that $U$ is a \emph{$\VIC$-submodule} of $V$. A $\VIC$-module $V$ is said to be \emph{generated by} a subset $S \subset \bigsqcup_{n\geqslant 0} V(\F_q^n)$ if the only $\VIC$-submodule of $V$ containing $S$ is $V$; we say that $V$ is \emph{finitely generated} if it is generated by a finite subset $S$. 

It is plain (see, for example, \cite[Lemma 2.14]{GL-Noetherian}) that a $\VIC$-module $V$ is finitely generated if and only if there exists a surjective homomorphism 
\begin{equation*}
P(m_1) \oplus \cdots \oplus P(m_k) \longrightarrow V \quad \mbox{ for some } m_1,\ldots,m_k \geqslant 0.
\end{equation*}
It is also easy to see that one has:

\begin{proposition} \label{fg iff RS2}
Let $V$ be a $\VIC$-module and let $\{V_n,\phi_n\}$ be the consistent sequence obtained from $V$. Then $V$ is finitely generated if and only if $\{V_n, \phi_n\}$ satisfies condition (RS2) and $\dim(V_n)<\infty$ for every $n$.
\end{proposition}
\begin{proof}
See \cite[Proposition 5.2]{GL-Noetherian}.
\end{proof}

Let us mention that it was also shown in \cite[Theorem E]{PS} that if $V$ is finitely generated, then $\{V_n, \phi_n\}$ satisfies condition (RS2).

\subsection{Weak stability and weight boundedness}
Suppose $\{V_n, \phi_n\}$ is a consistent sequence of representations of $G_n$. 

For any non-negative integers $n$, $m$ and $r$ such that $n=m+r$, the map $\phi_n: V_n \to V_{n+1}$ descends to a map
\begin{equation} \label{map on coinvariants}
\phi_{m,r} : (V_n)_{H_{m,r}} \longrightarrow (V_{n+1})_{H_{m,r+1}}.
\end{equation}

\begin{definition} \label{weakly stable}
We say that $\{V_n, \phi_n\}$ is:

\begin{itemize}
\item
\emph{weakly stable} if for each $m\in\Z_+$, there exists $s\geqslant 0$ such that for each $r\geqslant s$, the map $\phi_{m,r}$ of (\ref{map on coinvariants}) is bijective.

\item 
\emph{weight bounded} if there exists a non-negative integer $a$ such that for every $n\geqslant 0$ and every irreducible subrepresentation $\varphi(\la[n])$ of $V_n$, one has $\|\la\|\leqslant a$.
\end{itemize}
\end{definition}

The following key proposition is proved in \cite{GW}.

\begin{proposition} \label{RS1 and RS3}
Let $\{V_n, \phi_n\}$ be a consistent sequence of representations of $G_n$. Suppose that $\{V_n, \phi_n\}$ is weakly stable and weight bounded. Then:
\begin{itemize}
\item[(i)]
$\{V_n,\phi_n\}$ satisfies conditions (RS1) and (RS3). 

\item[(ii)]
If $\dim(V_n)<\infty$ for every $n\geqslant 0$, then there exists a polynomial $P\in \mathbb{Q}[T]$ and an integer $N$ such that $\dim(V_n) = P(q^n)$ for all $n\geqslant N$.
\end{itemize}
\end{proposition}
\begin{proof}
(i) See \cite[Proposition 4.1, Proposition 4.3]{GW}.

(ii) Immediate from condition (RS3) and \cite[Proposition 5.2]{GW}.
\end{proof}

\subsection{Proof of Theorem \ref{stability for VIC modules}}.
We now prove Theorem \ref{stability for VIC modules}. 

Let $V$ be a finitely generated $\VIC$-module and $\{V_n, \phi_n\}$ the consistent sequence obtained from $V$. By Proposition \ref{fg iff RS2} and Proposition \ref{RS1 and RS3}, it suffices to prove that $\{V_n, \phi_n\}$ is weakly stable and weight bounded. We need the following lemma.

\begin{lemma} \label{P weakly stable}
For any non-negative integer $m$, the $\VIC$-module $P(m)$ is weakly stable.
\end{lemma}

\begin{proof}
Write $\{P(m)_n, \phi_n\}$ for the consistent sequence obtained from $P(m)$. 

Fix a non-negative integer $\ell$ and set $s=m+\min(m,\ell)$. We shall prove that for $r \geqslant s$, the map
\begin{equation*}
\phi_{\ell,r} : (P(m)_n)_{H_{\ell,r}} \longrightarrow (P(m)_{n+1})_{H_{\ell,r+1}} \quad \mbox{ (where $n=\ell+r$) } 
\end{equation*}
is surjective. This would imply that $\phi_{\ell,r}$ is bijective for all $r$ sufficiently large.

Observe that for $n\geqslant m$, the group $G_n$ acts transitively on $\Hom_{\VIC}(\F_q^{\,m}, \F_q^{\,n})$ and $L_{m,n-m}$ is the stabilizer of the pair $(f, K)$ where $f$ is the standard inclusion $\F_q^{\,m} \hookrightarrow \F_q^{\,n}$ and $K$ is the subspace of vectors in $\F_q^{\,n}$ whose first $m$ coordinates are 0. Hence, it suffices to prove that for $r \geqslant s$, the map
\begin{equation*}
H_{\ell,r}\backslash G_n / L_{m,n-m} \longrightarrow H_{\ell,r+1}\backslash G_{n+1} / L_{m,n+1-m}
\end{equation*}
induced by the standard inclusion $G_n \hookrightarrow G_{n+1}$ is surjective. We shall prove the stronger statement that for $r\geqslant s$, the map
\begin{equation*}
L_{\ell,r}\backslash G_n / L_{m,n-m} \longrightarrow L_{\ell,r+1}\backslash G_{n+1} / L_{m,n+1-m}
\end{equation*}
is surjective. Thus, suppose that $r\geqslant s$ and choose any $x\in G_{n+1}$. We need to show that for some $g\in L_{\ell, r+1}$ and $g'\in L_{m,n+1-m}$, the only nonzero entry in the last row or last column of $gxg'$ is the entry in position $(n+1,n+1)$ and it is equal to $1$. The argument below is taken from \cite[Example 3.12]{GL-Noetherian}.

First, we may choose $g_1\in L_{\ell, r+1}$ and $g'_1\in L_{m,n+1-m}$ such that the entries of $g_1 x g'_1$ in positions $(i,j)$ are 0 if $i>m+\ell$ and $j\leqslant m$, or if $i\leqslant \ell$ and $j>m+\ell$. Indeed, we may first perform row operations on the last $r+1$ rows of $x$ to change the entries in positions $(i,j)$ to 0 for $i>m+\ell$ and $j\leqslant m$; then, we may perform column operations on the last $n+1-m$ columns to change the entries in positions $(i,j)$ to 0 for $i\leqslant \ell$ and $j>m+\ell$. Let $x_1 = g_1 x g'_1$.

Next, since $x_1$ has rank $n+1 > m+\ell+\min(m,\ell)$, there exists $i, j>m+\ell$ such that the entry of $x_1$ in position $(i,j)$ is nonzero. By first interchanging row $i$ with row $n+1$, and then interchanging column $j$ with column $n+1$, we obtain a matrix whose entry in position $(n+1, n+1)$ is nonzero. It is now easy to see that for some $g_2\in L_{\ell, r+1}$ and $g'_2\in L_{m,n+1-m}$, the matrix $g_2 x_1 g'_2$ lies in $G_n$.
\end{proof}

We can now prove that the consistent sequence $\{V_n, \phi_n\}$ is weakly stable. Since $V$ is finitely generated, there is a surjective homomorphism $P \to V$ where $P$ is a $\VIC$-module of the form $P(m_1)\oplus\cdots\oplus P(m_k)$ for some $m_1,\ldots,m_k\geqslant 0$. Let $\{P_n, \phi_n\}$ be the consistent sequence obtained from $P$. By Lemma \ref{P weakly stable}, $\{P_n, \phi_n\}$ is weakly stable. Fix $m\geqslant 0$. Then for all $r$ sufficently large, the top horizontal map in the following commuting diagram is bijective:
\begin{equation*}
\xymatrix{
(P_{m+r})_{H_{m,r}} \ar[rr]^-{\phi_{m,r}} \ar[d] && (P_{m+r+1})_{H_{m,r+1}} \ar[d] \\
(V_{m+r})_{H_{m,r}} \ar[rr]_-{\phi_{m,r}} && (V_{m+r+1})_{H_{m,r+1}}
}
\end{equation*}
Since the two vertical maps in the above diagram are surjective, it follows that the bottom horizontal map is surjective for all $r$ sufficiently large, and hence bijective for all $r$ sufficiently large.

Since $V$ is a quotient of $P(m_1)\oplus\cdots\oplus P(m_k)$, to see that $\{V_n, \phi_n\}$ is weight bounded, it suffices to show that $P(m)$ is weight bounded for each $m\geqslant 0$. As we noted in the proof of Lemma \ref{P weakly stable}, for each $n\geqslant m$, the $G_n$ representation $P(m)(\F_q^{\,n})$ is isomorphic to the permutation representation of $G_n$ on $G_n/L_{m,n-m}$. But $L_{m,n-m}$ is conjugate to $G_{n-m}$ in $G_n$. Therefore, the permutation representation of $G_n$ on $G_n/L_{m,n-m}$ is isomorphic to the permutation representation of $G_n$ on $G_n/G_{n-m}$. We deduce from Theorem \ref{main theorem} that the consistent sequence obtained from $P(m)$ is weight bounded. This completes the proof of Theorem \ref{stability for VIC modules}.

\subsection{Remarks} 
We end with some remarks on the connection of this paper to other works.

(i) It is a difficult result \cite[Theorem C]{PS} of Putman and Sam that any finitely generated $\VIC$-module over a noetherian ring is noetherian. An easier proof for finitely generated $\VIC$-modules over a field of characteristic zero is given in \cite[Example 3.12]{GL-Noetherian}. It was shown in \cite[Proposition 5.1]{GL-Noetherian} and \cite[Theorem E]{PS} that condition (RS1) is also a simple consequence of the noetherian property. Let us also mention that \cite{PS} defined a more general version of the category $\VIC$, called $\VIC(R, \mathfrak{U})$; they showed that the homology of congruence subgroups of general linear groups and automorphism groups of free groups form finitely generated $\VIC(R, \mathfrak{U})$-modules for suitable $R$ and $\mathfrak{U}$. 

(ii) Let $\VI$ be the category whose objects are the finite dimensional vector spaces over $\F_q$ and whose morphisms are the injective linear maps. In \cite[Theorem 1.6]{GW}, the authors proved a representation stability theorem for finitely generated $\VI$-modules over $\kk$. This result can also be deduced from Theorem \ref{stability for VIC modules}: there is a forgetful functor from the category $\VIC$ to the category $\VI$, and the pullback of any finitely generated $\VI$-module is a finitely generated $\VIC$-module; see \cite[Remark 1.13]{PS}. The proof of \cite[Theorem 1.6]{GW}, however, does not require the use of the branching rule (Fact \ref{branching rule for 1 step}) for the finite general linear groups.

(iii) We should mention that the categories $\VI$ and $\VIC$ are analogues of the category $\FI$ of finite sets and injective maps. $\FI$-modules were introduced and studied by Church, Ellenberg and Farb in \cite{CEF}. One of the main results of their paper is a representation stability theorem for finitely generated $\FI$-modules over a field of characteristic zero. The proof of Proposition \ref{RS1 and RS3}(i) (given in \cite[Proposition 4.1 and Proposition 4.3]{GW}) is an adaptation of their arguments for consistent sequences of representations of the symmetric groups to our situation of consistent sequences for representations of the finite general linear groups. Let us also mention that \cite{CEF} also proved eventual polynomiality of characters for finitely generated $\FI$-modules over a field of characteristic zero. We do not know if there are analogous results for $\VI$-modules or $\VIC$-modules.

\end{document}